\documentclass[11pt]{amsart}
\usepackage{amsfonts}
\usepackage{mathrsfs}
\usepackage{hyperref}
\usepackage[T1]{fontenc}

\newcommand{\FF}{\mathbb{F}}
\newcommand{\GG}{\mathbb{G}}

\newcommand{\NN}{\mathbb{N}}

\newcommand{\ZZ}{\mathbb{Z}}
\newcommand{\BB}{\mathscr{B}}
\newcommand{\KK}{\mathscr{K}}

\newcommand{\cC}{\mathcal{C}}

\newcommand{\cE}{\mathcal{E}}
\newcommand{\cL}{\mathcal{L}}

\newcommand{\cG}{\mathcal{G}}

\newcommand{\cM}{\mathcal{M}}
\newcommand{\cN}{\mathcal{N}}

\newcommand{\cU}{\mathcal{U}}

\newcommand{\cX}{\mathcal{X}}

\newcommand{\fg}{\mathfrak{g}}

\newcommand{\fn}{\mathfrak{n}}

\newcommand{\fu}{\mathfrak{u}}

\newcommand{\lra}{\longrightarrow}

\DeclareMathOperator{\ad}{ad}

\DeclareMathOperator{\Char}{char}

\DeclareMathOperator{\cx}{cx}

\DeclareMathOperator{\Der}{Der}

\DeclareMathOperator{\Kr}{Kr}
\DeclareMathOperator{\Ext}{Ext}

\DeclareMathOperator{\HH}{H}

\DeclareMathOperator{\rad}{rad}

\DeclareMathOperator{\Spec}{Spec}

\DeclareMathOperator{\IDer}{IDer}
\usepackage{pictex}
\usepackage{amscd}
\usepackage{latexsym}
\usepackage{amssymb}
\usepackage{eucal}
\usepackage{eufrak}
\usepackage{verbatim}
\usepackage{hyperref}
\numberwithin{equation}{section}
\newtheorem{Theorem}{Theorem}[section]
\newtheorem{Lemma}[Theorem]{Lemma}
\newtheorem{Corollary}[Theorem]{Corollary}
\newtheorem{Proposition}[Theorem]{Proposition}

\theoremstyle{Theorem}

\newtheorem*{thm*}{Theorem}
\newtheorem*{thm**}{Corollary}
\newtheorem*{thm***}{Theorem B}
\theoremstyle{remark}
\newtheorem*{Remark}{Remark}

\numberwithin{equation}{section}

\begin{document}

\title[First Hochschild]{On the first Hochschild cohomology of cocommutative Hopf algebras of finite representation type}
\author[Hao Chang]{Hao Chang}
\address[Hao Chang]{School of Mathematics and Statistics,
Central China Normal University, 430079 Wuhan, People's Republic of China}
\email{chang@mail.ccnu.edu.cn}
\date{\today}
\begin{abstract}
Let $\BB_0(\cG)\subseteq k\cG$ be the principal block algebra of the group algebra $k\cG$ of an infinitesimal group scheme $\cG$ over an algebraically closed field
$k$ of characteristic $\Char(k)=:p\geq 3$.
We calculate the restricted Lie algebra structure of the first Hochschild cohomology $\cL:=\HH^1(\BB_0(\cG),\BB_0(\cG))$ whenever $\BB_0(\cG)$ has finite representation type.
As a consequence, we prove that the complexity of the trivial $\cG$-module $k$ coincides with the maximal toral rank of $\cL$.
\end{abstract}
\maketitle

\section*{Introduction}
Let $k$ be an algebraically closed field and $A$ be a finite-dimensional $k$-algebra.
The first Hochschild cohomology group $\cL:=\HH^1(A,A)$ can be identified with the space of outer derivations of $A$,
and the commutator of derivations endow this space with a Lie algebra structure.
The Lie structure of $\cL$ has been studied in several recent article, see for instance \cite{LR2}, \cite{RSS}, \cite{ER}.

Suppose that the field $k$ is of odd prime characteristic $p$.
For an infinitesimal group scheme $\cG$,
we will be interested in the representation theory of the principal block algebra $\BB_0(\cG)$ of the group algebra $k\cG$.
The main motivation for the note is that the Lie algebra $\HH^1(\BB_0(\cG),\BB_0(\cG))$ has the $p$-restricted structure \cite{Zimm}.
The purpose of this paper is to describe the close connections between the representations of $\BB_0(\cG)$ and the restricted Lie algebra structure of $\HH^1(\BB_0(\cG),\BB_0(\cG))$. We denote by $\cx_\cG(k)$ the complexity of the trivial $\cG$-module $k$.

\begin{thm*}
Let $\cG$ be an infinitesimal group scheme of finite representation type.
Then the restricted Lie algebra $\cL:=\HH^1(\BB_0(\cG),\BB_0(\cG))$ is trigonalizable,
and $\cx_\cG(k)=1$ coincides with the maximal toral rank of $\cL$.
\end{thm*}

For the reduced group scheme, the first Hochschild cohomology of the block algebra has been studied in \cite{LR1}.
They showed that the only restricted simple Lie algebra occurs as $\HH^1(B,B)$ of some block algebra $B$ of a finite group is
the Jocobson-Witt algebra.

After recalling basic definitions on the finite group schemes, the first Hochschild cohomology and the representation type.
We show in Section 2 that the Lie algebra $\HH^1(k\cU,k\cU)$ associated with a unipotent group scheme is a simple Lie algebra if and only if $\cU$ is elementary abelian.
This can be thought of as a generalization of \cite[Proposition 2.5]{LR1}.
Section 3 is concerned with the restricted Lie algebra structure of $\HH^1(\BB_0(\cG),\BB_0(\cG))$ for the infinitesimal group scheme of finite representation type.
Finally we will apply our results to the blocks of Frobenius kernels of smooth groups.

\noindent
\emph{Throughout this paper, all vector spaces are assumed to be finite-dimensional over a fixed algebraically closed field $k$ of characteristic $\Char(k):=p\geq 3$}.

\section{Preliminaries}
\subsection{Finite group schemes}\label{section finite group schemes}
Throughout this paper, $\cG$ is assumed to be a finite group scheme,
defined over $k$.
We denote by $k[\cG]$ and $k\cG:=k[\cG]^{*}$, the coordinate ring and the
group algebra (the algebra of measures) of $\cG$, respectively.
$\cG$ is called \textit{infinitesimal} if $k[\cG]$ is local.
A finite group scheme $\cU$ is said to be \textit{unipotent},
provided $k\cU$ is a local algebra.
The reader is referred to \cite{Ja06} and \cite{Wa79} for basic facts concerning algebraic group schemes.

Given a finite-dimensional restricted Lie algebra $(\fg,[p])$,
we let $U_0(\fg)$ denote the \textit{restricted enveloping algebra} of $\fg$,
i.e., the factor algebra of the universal enveloping algebra $U(\fg)$ by the ideal generated by $\{x^p-x^{[p]};~x\in\fg\}$.
Since the dual algebra $U_0(\fg)^{*}$ is local,
$\cG_\fg:=\Spec(U_0(\fg)^{*})$ is an infinitesimal group scheme of height $1$ such that $k\cG_\fg=U_0(\fg)$.
Let $(\fg,[p])$ be a restricted Lie algebra.
Recall that
$$V(\fg):=\{x\in\fg;~x^{[p]}=0\}$$
is the \textit{restricted (nullcone)} of $\fg$.

Following \cite{Fa17}, we refer to an abelian finite group scheme $\cE$ as \textit{elementary abelian},
provided there exist subgroup schemes $\cE_1, \ldots, \cE_n \subseteq \cE$ such that
\begin{enumerate}
\item[(a)] $\cE = \cE_1\cdots \cE_n$, and
\item[(b)] $\cE_i \cong \GG_{a(r_i)}, \ZZ/(p)$.
\end{enumerate}
Here $\GG_{a(r)}$ denotes the $r$-th Frobenius kernel of the additive group $\GG_a = \Spec_k(k[T])$,
while $\ZZ/(p)$ refers to the reduced group scheme,
whose group of $k$-rational points is the cyclic group $\ZZ/(p)$.
In particular,
a restricted Lie algebra $(\fg,[p])$ is \textit{elementary abelian} if $\fg$ is abelian and $\fg=V(\fg)$,
and $(\fg,[p])$ is $p$-\textit{nilpotent} if every $x\in\fg$ is annihilated by some iterate of the $p$-map.

\subsection{The first Hochschild cohomology $\HH^1(A,A)$}
Let $A$ be a finite-dimensional $k$-algebra.
A \textit{derivation} on A is a $k$-linear map $f: A\rightarrow A$ satisfying $f(ab)=f(a)b+af(b)$
for all $a, b\in A$.
For any $a\in A,~n\in\NN$ we have $f(a^n)=\sum_{i=1}^na^{i-1}f(a)a^{n-i}$.
If $f, g$ are derivations on $A$,
then the bracket $[f,g]:=f\circ g-g\circ f$ is also a derivation,
and the space $\Der(A)$
of derivations on $A$ becomes a Lie algebra.
If $x\in A$, then the map $\ad x$ defined by $\ad x(a):=xa-ax$ is a derivation,
any derivation of this form is called an \textit{inner derivation}.
The space $\IDer(A)$ of inner derivations is a Lie ideal in $\Der(A)$,
we have a canonical identification $\HH^1(A,A)\cong\Der(A)/\IDer(A)$,
where $\HH^1(A,A)$ is the first degree Hochschild cohomology of $A$.
The reader is referred to \cite[Chapter 9]{Wei} for basic
facts concerning Hochschild cohomology.
Note that the Hochschild cohomology
is Morita invariant as a Gerstenhaber algebra (see \cite{GS}) and
$\HH^1(A,A)$ is a restricted Lie algebra, and this is a derived invariant \cite{Zimm}.

\subsection{Representation type}
According to a fundamental theorem by Drozd \cite{Drozd},
the category of finite-dimensional algebras over an algebraically closed field may be subdivided into the disjoint classes of representation-finite,
tame and wild algebras.
Recall that a finite-dimensional associative algebra $A$ has \textit {finite representation type}
if the set of isoclasses of finite-dimensional indecomposable $A$-modules is finite.
We say that $A$ is \textit{tame} if for each natural number $d$ the set of isoclasses of $d$-dimensional
indecomposable $A$-modules is the union of finitely many one-parameter families and a finite set.
All other algebras are of \textit{wild representation type}.
We will employ \cite{AUS} and \cite{Erd} as general references on these matters.

\section{Unipotent group scheme}
Given a finite-dimensional algebra $A$, we always denote by $J(A)$ the Jacobson radical of $A$.
Let $\cU$ be a unipotent group scheme.
By definition, the group algebra $k\cU$ is local.
\begin{Lemma}\label{kU is local symmetric}
Let $\cU$ be a unipotent group scheme.
Then the group algebra $k\cU$ is a local symmetric $k$-algebra.
In particular, $[k\cU,k\cU]\subseteq J(k\cU)^2$.
\end{Lemma}
\begin{proof}
It is well known (cf. \cite[(1.5)]{FMS}) that
the Nakayama automorphism of $k\cU$ is given by the convolution ${\rm id}_{k\cU}\ast\zeta$,
where $\zeta$ is the modular function.
Since $k\cU$ is local, $\zeta$ is the co-unit,
so that the Nakayama automorphism is the identity.
Hence $k\cU$ is isomorphic to its dual $A^{*}$ as an $A-A$-bimodule,
so that $k\cU$ is symmetric.
On the other hand, every element in $k\cU$ is of the form $\lambda+a$ for some $\lambda\in k$ and $a\in J(k\cU)$,
this yields $[k\cU,k\cU]\subseteq J(k\cU)^2$.
\end{proof}

Given a restricted Lie algebra $(\fg,[p])$.
The maximum dimension of all tori of $\fg$ will be denoted by $\mu(\fg)$.

Suppose that $\cE$ is an elementary abelian group scheme.
In view of \cite[(6.2.1),(6.2.2)]{Fa17}, there is an isomorphism:
$$k\cE\cong k[T_1,\ldots,T_n]/(T_i^p;~1\leq i\leq n):=B_n.$$
It follows that $\HH^1(k\cU,k\cU)=\Der(B_n)=:W_n$ is a restricted simple Lie algebra
(called $n$-th Jaconsbon-Witt algebra, see \cite[(IV.4.2)]{SF} for details).
Note that $W_n$ admits a self-centralizing torus and  $\mu(W_n)=n$ (cf. \cite[(IV.2.5)]{SF}).

Given a $\cG$-module $M$, we denote by $\cx_{\cG}(M)$ the complexity of $M$.
\begin{Proposition}\label{abelian unipotent group}
Let $\cU$ be an abelian unipotent group and $\cL:=\HH^1(k\cU,k\cU)$.
Then exists a $p$-nilpotent ideal $\fn$ such that $\cL/\fn\cong W_n$.
In particular, $\mu(\cL)=n=\cx_\cU(k)$.
\end{Proposition}
\begin{proof}
As $\cU$ is abelian unipotent, general theory (\cite[(14.4)]{Wa79}) provides an isomorphism
$$k\cU\cong k[X_1,\ldots,X_n]/(X_1^{p^{a_1}},\ldots,X_n^{p^{a_n}}),$$
where $a_i\in\NN$ and $\cx_\cU(k)=n$.
For $i\in\{1,\ldots,n\}$, let $x_i=X_i+(X_1^{p^{a_1}},\ldots,X_n^{p^{a_n}})\in k\cU$,
and denote by $\partial_i:=\partial/\partial x_i$ the partial derivative with respect to the variable $x_i$.
It is easy to verify that $\cL=\{\sum_{i=1}^n f_i\partial_i;~f_i\in k\cU\}$.
There is a natural restricted $\mathbb{Z}$-grading on $\cL$ (see for instance \cite[(III.2)]{SF}), i.e.,
$\cL=\sum_i\cL_{[i]}$, where $\cL_{[i]}=\sum_{i=1}^n\sum_{\alpha_1+\cdots+\alpha_n=i+1} kx_1^{\alpha_1}\cdots x_n^{\alpha_n}\partial_i$.
Now let $\fn$ be subspace generated by the following set
$$\{x_1^{\alpha_1}\cdots x_n^{\alpha_n}\partial_k;~1\leq k\leq n;~a_i\geq p~{\rm for~some}~i\}.$$
Direct computation shows that $\fn$ is an ideal of $\cL$.
The restricted $\mathbb{Z}$-grading on $\fn$ ensures that the ideal is $p$-nilpotent.
Hence, $\mu(\fn)=0$ (cf. \cite[(3.3)]{FV03}).
Moreover,
the factor algebra $\cL/\fn$ is obviously isomorphic to the $n$-th Jacobson-Witt algebra $W_n$.
Consequently, $\mu(\cL)=\mu(W_n)=n$.
\end{proof}

The following result can be thought of as a generalization of \cite[Proposition 2.5, Remark 2.6]{LR1},
and the proof is adapted from \cite[Section 4]{LR1}.
\begin{Proposition}\label{simple iff elementary abelian}
Let $\cU$ be a unipotent group scheme.
Then $\cL:=\HH^1(k\cU,k\cU)$ is a simple Lie algebra if and only if $\cU$ is elementary abelian.
\end{Proposition}
\begin{proof}
According to Lemma \ref{kU is local symmetric}, $k\cU$ is a local symmetric algebra.
Suppose that $\cL$ is a simple Lie algebra.

We first show that $\cU$ is an abelian unipotent group.
Let $Z(k\cU)$ be the center of $k\cU$.
Suppose that $Z(k\cU)\neq k\cU$,
the Nakayama Lemma yields $J(Z(k\cU))k\cU\neq J(k\cU)$.
In view of \cite[(3.1), (3.2)]{LR1},
the canonical Lie algebra homomorphism $\cL\rightarrow\HH^1(Z(k\cU),Z(k\cU))$ is not injective.
By assumption, the homomorphism is zero.
Note that the algebra $k\cU/J(k\cU)^2$ is commutative (Lemma \ref{kU is local symmetric}).
We thus obtain a contraction by using the same argument in \cite[Page 1127]{LR1}.

As $\cU$ is abelian unipotent, the proof of Proposition \ref{abelian unipotent group}
implies that $\cU$ must be elementary abelian.
\end{proof}

Let $\cG$ be a finite group scheme,
we will denote the principal block of $k\cG$ by $\BB_0(\cG)$ and let $\cG_{\rm lr}$ be the largest linearly reductive normal subgroup of $\cG$,
see \cite[(I.2.37)]{Vo}. If $\cG$ is a reduced finite group scheme,
then $\cG_{\rm lr}$ corresponds to the largest normal subgroup $O_{p'}(\cG(k))$ of the finite group $\cG(k)$, whose order is prime to $p$.

\begin{Corollary}\label{nilpotent principle block}
If $\cN$ is a nilpotent group scheme,
then $\cL:=\HH^1(\BB_0(\cN),\BB_0(\cN))$ is simple if and only if $\cN/\cN_{\rm lr}$ is elementary abelian.
\end{Corollary}
\begin{proof}
General theory tells us that $\cN_{\rm lr}$ is the unique largest
diagonalizable subgroup scheme of $\cN$ and $\cN/\cN_{\rm lr}$ is unipotent.
According to \cite[(1.1)]{Fa06},
there is an isomorphism $\BB_0(\cN)\cong\BB_0(\cN/\cN_{\rm lr})$ between the principal blocks
of $k\cN$ and $k(\cN/\cN_{\rm lr})$.
The assumption in conjunction with Proposition \ref{simple iff elementary abelian} yields the assertion.
\end{proof}

\section{Infinitesimal groups of finite representation type}
\subsection{$\BB_0(\cG)$ has finite representation type}
Given an infinitesimal group $\cG$,
we denote by $\cM(\cG)$ the unique largest diagonalizable (multiplicative) normal subgroup of $\cG$.
In view of \cite[(7.7),(9.5)]{Wa79},
the group scheme $\cM(\cG)$ coincides with the multiplicative constituent of the center $\cC(\cG)$ of $\cG$.

For an arbitrary finite group scheme $\cG=\Spec_k(k[\cG])$,
we let $\cG^{(1)} =\Spec_k(k[\cG]^{(1)})$ be the finite group scheme,
whose coordinate ring differs from $k[\cG]$ only by its structure as a $k$-space: On $k[\cG]^{(1)}$ an element $\alpha \in k$ acts via $\alpha^{p^{-1}}$, cf.\ \cite[(I.9.2)]{Ja06}.
In \cite{DG70}, this group scheme is denoted $\cG^{(p)}$.

For a commutative unipotent infinitesimal group scheme $\cU$,
we let
\[ V_{\cU} : \cU^{(p)} \lra \cU\]
be the Verschiebung, cf.\ \cite[(IV,\S3,${\rm n}^{\rm o}$4),(II,\S7,${\rm n}^{\rm o}$1)]{DG70}. Following \cite{FV00}, we refer to
$\cU$ as being \textit{V-uniserial}, provided there is an exact sequence
\[\cU^{(p)} \stackrel{V_{\cU}}{\lra} \cU \lra \GG_{a(1)} \lra e_k.\]

It is well known (cf. \cite[(3.1)]{FV00}) the group algebra $k\cG$
has finite representation type if and only if its principal block $\BB_0(\cG)$ enjoys this property.
The $r$-th Frobenius kernel of the multiplicative group $\GG_m := \Spec_k(k[X,X^{-1}])$ will be denoted $\GG_{m(r)}$.
Given an infinitesimal group $\cG$ of finite representation type,
it was shown in \cite[(2.7)]{FV00} that
$$\cG':=\cG/\cM(\cG)\cong\cU\rtimes\GG_{m(r)}$$
is a semidirect product with a V-uniserial normal subgroup $\cU$.

In the sequel, we will employ the \textit{adjoint representation} of $k\GG_{m(r)}$ on $k\cU$.
Letting $\eta$ denote the antipode of the Hopf algebra $k\GG_{m(r)}$, this action is given by
\begin{align}\label{module action}
m\cdot u:=\sum\limits_{(m)}m_{(1)}u\eta(m_{(2)}),~\forall~m\in k\GG_{m(r)},~u\in k\cU.
\end{align}
The adjoint representation endows $k\cU$ with the structure of an $k\GG_{m(r)}$-module algebra,
and the multiplication $k\cU\otimes_k k\GG_{m(r)}\rightarrow k\cG'$ induces
an isomorphism between the smash product $k\cU\#k\GG_{m(r)}$ and $k\cG'$ (cf. \cite[p.40f]{Mo}).
The character group of $k\GG_{m(r)}$ will be denoted $\cX:=\cX(\GG_{m(r)})$.
We will write the convolution product on $\cX(\GG_{m(r)})$ additively, i.e.,
\begin{align}\label{additive structure of character groups}
(\lambda+\mu)(m)=\sum\limits_{(m)}\lambda(m_{(1)})\mu(m_{(2)}),~\forall~m\in k\GG_{m(r)},~\lambda, \mu\in\cX(\GG_{m(r)}).
\end{align}
Observe that the augmentation is the zero element and $\cX\cong\ZZ/(p^r)$.

\subsection{Restricted Lie algebra structure}
In the sequel, we will make a detailed examination of the restricted Lie algebra structure $\HH^1(A,A)$.
We will always assume that $\cG$ is an infinitesimal group of finite representation type,
so that $\cG'=\cG/\cM(\cG)\cong\cU\rtimes\GG_{m(r)}$ is a semidirect product with a $V$-uniserial normal subgroup $\cU$.

Since $k\GG_{m(r)}$ is semisimple,
there results a weight space decomposition
$$k\cU=\bigoplus\limits_{\alpha\in R}(k\cU)_\alpha,$$
where $R$ is a finite subset of $\cX$.
There exists a root vector $x\in (k\cU)_\alpha$ such that $\{1, x, x^2,\ldots, x^{p^n-1}\}$
is a basis of $k\cU$ (cf. \cite[p. 13]{FV00}).
Thus,
$$R=\{i\alpha;~0\leq i\leq p^n-1\}.$$
Moreover, the commutative semisimple algebra $k\GG_{m(r)}$ decomposes into
a direct sum $$k\GG_{m(r)}=\bigoplus\limits_{\lambda\in\cX(\GG_{m(r)})}ku_\lambda,$$
with primitive orthogonal idempotents $u_\lambda$.
Accordingly, we have $mu_\lambda=\lambda(m)u_\lambda$
for every element $m\in k\GG_{m(r)}$,
so that $\lambda(u_\mu)=\delta_{\mu,\lambda},~\forall~\lambda, \mu\in\cX(\GG_{m(r)})$.

Now, we let $A:=k\cG'=k\cU\#k\GG_{m(r)}$. Then $A$ has the set of monomials
\begin{align}\label{basis of A}
\mathfrak{B}:=\{u_\lambda x^j;~\lambda\in\cX,~0\leq j\leq p^n-1\}
\end{align}
as a $k$-basis.
Let $f\in\Der(A)$,
then $f$ is uniquely determined by its values at $u_\lambda$ ($\lambda\in\cX$) and $x$.
\begin{Lemma}\label{multiplication rules in kG'}
With the above notations.
Then we have $u_\lambda x u_\mu=\delta_{\lambda,\mu+\alpha}xu_\mu$.
In particular, $u_\lambda x=xu_{\lambda-\alpha}$.
\end{Lemma}
\begin{proof}
By definitions ((\ref{module action}) and (\ref{additive structure of character groups})),
we obtain
\begin{align*}
u_\lambda xu_\mu &=\sum\limits_{(u_\lambda)}(u_{\lambda,(1)}\cdot x)u_{\lambda,(2)}u_\mu=\sum\limits_{(u_\lambda)}\alpha(u_{\lambda,(1)})x\mu(u_{\lambda,(2)})u_\mu\\
                 &=(\mu+\alpha)(u_\lambda)xu_\mu=\delta_{\lambda,\mu+\alpha}xu_\mu.
\end{align*}
Moreover,
$u_\lambda x=u_\lambda x\sum_{\mu\in\cX}u_\mu=\sum_{\mu\in\cX}u_\lambda xu_\mu$,
our assertion follows.
\end{proof}

\begin{Lemma}\label{innder derivation of kG'}
For $\lambda, j$ such that $\lambda\in\cX(\GG_{m(r)}),~0\leq j\leq p^n-1$,
consider the inner derivation $d_{\lambda,j}=\ad u_\lambda x^j$ on $A$.
\begin{enumerate}
\item $d_{\lambda,j}(u_\mu)=(\delta_{\mu+j\alpha,\lambda}-\delta_{\mu,\lambda})u_\lambda x^j$. In particular, $d_{\lambda,j}(u_\mu)=0$ if $p^r\mid j$.
\item $d_{\lambda,j}(x)=(u_\lambda-u_{\lambda+\alpha})x^{j+1}$. In particular, $d_{\lambda,j}(x)=0$ if and only if $j=p^n-1$.
\item Let $d$ be an inner derivation.
Then $d(u_\mu)$ is a linear combination of monomials $u_\mu x^j$ and $u_{\mu+j\alpha}x^j$ with $0\leq j\leq p^n-1$ such that $p^r\nmid j$.
Similarly, $d(x)$ is a linear combination of $\sum_{i=0}^{p^r-1}a_iu_{i\alpha}x^j$ with $\sum_{i=0}^{p^r-1}a_i=0,~a_i\in k;~1\leq j\leq p^n-1$.
\end{enumerate}
\end{Lemma}
\begin{proof}
Note that $\alpha$ has order $p^r$ (\cite[(2.1)]{FV00}),
in view of Lemma \ref{multiplication rules in kG'},
we have $u_{\mu+\alpha}x=xu_\mu$.
It follows that
$$d_{\lambda,j}(u_\mu)=[u_\lambda x^j,u_\mu]=u_\lambda u_{\mu+j\alpha}x^j-u_\mu u_\lambda x^j=(\delta_{\mu+j\alpha,\lambda}-\delta_{\mu,\lambda})u_\lambda x^j.$$
Similarly,
$$d_{\lambda,j}(x)=[u_\lambda x^j,x]=u_\lambda x^{j+1}-xu_\lambda x^j=u_\lambda x^{j+1}-u_{\lambda+\alpha}x^{j+1},$$
this proves (2).
An inner derivation on $A$ is a linear combination of the inner derivations $d_{\lambda,j}$,
part (3) now follows from (1) and (2).
\end{proof}

The adjoint representation of $k\GG_{m(r)}$ in (\ref{module action}) can be extended naturally to the whole algebra $A$.
As $k\GG_{m(r)}$ is commutative,
the module $k\GG_{m(r)}$ belongs to the zero weight space.
\begin{Lemma}\label{derivation kG'}
Let $f\in\Der(A)$.
Then there is a derivation $g\in\Der(A)$ with the following properties:
\begin{enumerate}
\item[(a)] $g(k\GG_{m(r)})=0,~m\cdot g(x)=\alpha(m)g(x),~\forall~m\in k\GG_{m(r)}$, and
\item[(b)] $f\equiv g~\pmod{\IDer(A)}$.
\end{enumerate}
\end{Lemma}
\begin{proof}
Since $k\GG_{m(r)}$ is separable,
it follows that $\HH^1(k\GG_{m(r)}, A)=(0)$ (\cite[Theorem 4.1]{Ho}).
Thus, if $f\in\Der(A)$,
then the restriction $f|_{k\GG_{m(r)}}$ is an inner derivation.
Hence there is an element $a\in A$ such that $(f-\ad a)|_{k\GG_{m(r)}}=0$.
Set $g:=f-\ad a$, we apply (\ref{module action}) to see that
$$g(m\cdot x)=g(\sum\limits_{(m)}m_{(1)}x\eta(m_{(2)}))=\sum\limits_{(m)}m_{(1)}g(x)\eta(m_{(2)})=m\cdot g(x)=\alpha(m)g(x).$$
\end{proof}
As $A$ is a completely reducible $k\GG_{m(r)}$-module,
it follows that the set
\begin{align}\label{basis of weight alpha}
\mathfrak{J}:=\{u_\lambda x^{jp^r+1};~\lambda\in\cX,~0\leq jp^r+1\leq p^n-1\}
\end{align}
is a $k$-basis of the weight space of $A$ with weight $\alpha\in\cX$,
where $~0\leq j\leq p^{n-r}-1$ if $n\geq r$ and $j=0$ if $n<r$.

\begin{Lemma}\label{monomial derivation g_lambda j}
Let $\lambda\in\cX,~0\leq jp^r+1\leq p^n-1$.
There is a derivation $g_{\lambda,j}\in\Der(A)$ satisfying $g_{\lambda,j}(k\GG_{m(r)})=0$ and
$g_{\lambda,j}(x)=u_\lambda x^{jp^r+1}$.
\end{Lemma}
\begin{proof}
When combined with Lemma \ref{derivation kG'},
the foregoing result implies that $g_{\lambda,j}(m\cdot x)=\alpha(m)g_{\lambda,j}(x)$ for all $m\in k\GG_{m(r)}$.
It suffices to verify that $g_{\lambda,j}(x^{p^n})=0$.
According to Lemma \ref{multiplication rules in kG'} we have
\begin{align*}
g_{\lambda,j}(x^{p^n})&=\sum\limits_{k=1}^{p^n}x^{k-1}g_{\lambda,j}(x)x^{p^n-k}=\sum\limits_{k}x^{k-1}u_\lambda x^{p^n-k+jp^r+1}\\
 &=\sum\limits_{k}u_{\lambda+(k-1)\alpha}x^{k-1}x^{p^n-k+jp^r+1}=0.
\end{align*}
\end{proof}

We let $\cL:=\HH^1(A,A)\cong\Der(A)/\IDer(A)$.
We determine a linear independent subset of $\Der(A)$ whose images in $\cL$ is a $k$-basis.
\begin{Lemma}\label{outer derivation}
With the notation of Lemma \ref{monomial derivation g_lambda j}.
We set $H:=\{g_{0,j};~0\leq jp^r+1\leq p^n-1\}$.
Then the set $H$ is linearly independent,
and its span $\mathcal{H}$ is a complement of $\IDer(A)$ in $\Der(A)$.
\end{Lemma}
\begin{proof}
The linear independence of the set $H$ follows from the fact that the set $\mathfrak{B}$ is a basis of $A$ (see (\ref{basis of A})).
Let $f\in\Der(A)$.
According to Lemma \ref{derivation kG'} and Lemma \ref{monomial derivation g_lambda j},
there exists a derivation $g$ such that $g$ is a linear combination of $\{g_{\lambda,j};~\lambda\in\cX,~0\leq jp^r+1\leq p^n-1\}$
and $f\equiv g~\pmod{\IDer(A)}$.
Moreover, Lemma \ref{innder derivation of kG'}(1)(2) ensures that
the space
$$\{\sum\limits_{i=0}^{p^r-1}a_ig_{i\alpha,j};~\sum_ia_i=0,~a_i\in k\}$$ is contained in $\IDer(A)$
for every $j$ with $0\leq jp^r+1\leq p^n-1$.
This yields a derivation $g'\in\mathcal{H}$ such that $f\equiv g'~\pmod{\IDer(A)}$.
It follows from Lemma \ref{innder derivation of kG'}(3) that $\mathcal{H}$ intersects $\IDer(A)$ trivially.
Hence $\mathcal{H}$ is a complement of $\IDer(A)$ in $\Der(A)$.
\end{proof}

\begin{Lemma}\label{Lie bracket formula}
Let $g_{0,i},g_{0,j}\in H$.
Then $[g_{0,i},g_{0,j}]=(j-i)g_{0,i+j}$.
\end{Lemma}
\begin{proof}
Owing to Lemma \ref{monomial derivation g_lambda j},
the Lie bracket $[g_{0,i},g_{0,j}](k\GG_{m(r)})=0$ and
\begin{align*}
[g_{0,i},g_{0,j}](x)&=g_{0,i}(g_{0,j}(x))-g_{0,j}(g_{0,i}(x))\\
             &=g_{0,i}(u_0x^{jp^r+1})-g_{0,j}(u_0x^{ip^r+1})\\
             &=u_0(g_{0,i}(x^{jp^r+1})-g_{0,j}(x^{ip^r+1})\\
             &=u_0(\sum\limits_{k=1}^{jp^r+1}u_{(k-1)\alpha}x^{(i+j)p^r+1}-\sum\limits_{k=1}^{ip^r+1}u_{(k-1)\alpha}x^{(i+j)p^r+1})\\
             &=(j-i)u_0x^{(i+j)p^r+1}\\
             &=(j-i)g_{0,i+j}(x).
\end{align*}
The second to last equality follows from Lemma \ref{multiplication rules in kG'} and
the fact that $\{u_\lambda;~\lambda\in\cX\}$ are the
primitive orthogonal idempotents of $k\GG_{m(r)}$.
\end{proof}

In the last, we determine the restricted Lie algebra structure on $\cL$.
The proof is similar to Lemma \ref{Lie bracket formula} and will be omitted.
\begin{Lemma}\label{p-map}
Let $g_{0,j}\in H$.
Then $g_{0,0}^{[p]}=g_{0,0}$ and $g_{0,j}^{[p]}=0$ if $j\neq 0$.
\end{Lemma}

\begin{Remark}
(1). Lemma \ref{Lie bracket formula}and Lemma \ref{p-map} imply that the space
$\mathcal{H}$ is a restricted Lie subalgebra of $\Der(A)$ with
$\Der(A)=\IDer(A)\oplus\mathcal{H}$.\\
(2). We still denote by $g_{0,j}$ its image in $\cL$.
There is a natural restricted $\mathbb{Z}$-grading on $\cL$:
$\cL=\sum_{j\geq 0}\cL_{[j]}$, where $\cL_{[j]}=kg_{0,j}$.
\end{Remark}

We summarize our main results in this subsection as follows:
\begin{Theorem}\label{main theorem 1}
Let $\cG$ be an infinitesimal group G of finite representation type.
Then the restricted Lie algebra $\cL=\HH^1(\BB_0(\cG),\BB_0(\cG))$ is trigonalizable,
and the one-dimensional space $\cL_{[0]}$ is a maximal torus of $\cL$.
In particular, $\cx_\cG(k)=\mu(\cL)=1$.
\end{Theorem}
\begin{proof}
Recall that $\cG'=\cG/\cM(\cG)\cong\cU\rtimes\GG_{m(r)}$.
According to \cite[(2.2), (2.4)]{FV00}, there is an isomorphism $\BB_0(\cG)\cong k\cG'$.
The assertion now follows from Lemma \ref{Lie bracket formula} and Lemma \ref{p-map}
in conjunction with \cite[(2.7)]{FV00}.
\end{proof}

\subsection{Example}
Now, let $\fg:=kt\oplus(kx)_p$ be a restricted Lie algebra
with Lie bracket $[t,x]=x$ and the $p$-map $t^{[p]}=t$ and $x^{[p]^n}=0$ for some $0\neq n\in\NN$.
This is a representation-finite restricted Lie algebra (cf. \cite{Fa95}).
We denote by $\BB_0(\fg)$ the principal block of $U_0(\fg)$.

The following Lemma is a particular case Proposition \ref{abelian unipotent group}
\begin{Lemma}\label{h1 of kx/x^p^n}
Let $A:=k[X]/(X^{p^n})$ be the truncated polynomial ring and $\cL:=\HH^1(A,A)$.
Then exists a $p$-nilpotent ideal $\fn$ such that $\cL/\fn\cong W_1$.
In particular, $\mu(\cL)=1$.
\end{Lemma}

Recall that $\cG_\fg:=\Spec(U_0(\fg)^{*})$ is an infinitesimal group scheme of height $1$ such that $k\cG_\fg=U_0(\fg)$.
As a corollary, we readily obtain:
\begin{Corollary}\label{repfinite principal block Lie algebra}
Let $(\fg,[p])$ be a representation-finite restricted Lie algebra
and $\cL:=\HH^1(\BB_0(\fg),\BB_0(\fg))$. Then $\mu(\cL)=1$.
Moreover, the Lie algebra $\cL$ is solvable if and only if $\fg$ is not nilpotent.
\end{Corollary}
\begin{proof}
According to \cite[(4.3)]{Fa95},
there exist a toral element $t\in\fg$ and a $p$-nilpotent element $x\in N(\fg)$ such that
$\fg=kt\oplus N(\fg)=kt\oplus T(\fg)\oplus (kx)_p$ ($\ast$).\label{structure of finite Lie algebra}
Here, the toral element $t$ can be equal to $0$.

If $\fg$ is nilpotent,
then \cite[(4.3)(b)]{Fa95} and \cite[(1.1)]{Fa06} yield an isomorphism
$$\BB_0(\fg)\cong\BB_0((kx)_p)\cong U_0((kx)_p).$$
Since $x$ is $p$-nilpotent, $U_0((kx)_p)\cong k[X]/(X^{p^n})$ for some integer $n$.
Thanks to Lemma \ref{h1 of kx/x^p^n}, the Lie algebra $\HH^1(\BB_0(\fg))$ is not solvable.

If $\fg$ is not nilpotent, then $t\neq 0$ in ($\ast$).
Since $\fg$ is supersolvable (cf. \cite[(2.1)]{FV00}),
the factor Lie algebra $\fg':=\fg/T(\fg)$ is trigonalizable (cf. \cite[(2.3)]{FV00}).
Recall that $\BB_0(\fg)\cong\BB_0(\fg')$,
our assumption now implies that $\fg'\cong kt\oplus (kx)_p$.
Then Theorem \ref{main theorem 1} implies that $\cL$ is trigonalizable,
so that $\cL$ is solvable.

In the last,
Lemma \ref{h1 of kx/x^p^n} and Proposition \ref{main theorem 1} imply that $\mu(\cL)=1$.
\end{proof}
\begin{Remark}
Suppose that $\fg$ is not nilpotent.
A consecutive application of \cite[(2.2)]{Fa96} and \cite[(3.2)]{Fa95} shows that $\BB_0(\fg')=U_0(\fg')$ is a Nakayama algebra
and $\BB_0(\fg')$ possesses exactly $p$ simple one-dimensional modules $k_\lambda;~0\leq \lambda\leq p-1$.
We apply \cite[(3.2)]{Fa95} to see that
$\Ext_{U_0(\fg')}^1(k_\lambda,k_\lambda)=(0)$
and $\dim_k\Ext_{U_0(\fg')}^1(k_\lambda,k_\mu)\leq 1$ for any $\lambda,\mu\in\{0,1,\cdots,p-1\}$.
The solvability of $\cL$ also follows from \cite[(3.1)]{LR2}.
However, we can not determine the maximal toral rank $\mu(\cL)$ from the result in \cite{LR2}.
\end{Remark}

\section{Application to blocks of Frobenius kernels of smooth groups}
\subsection{Trivial extension}
We denote by $T(\Kr):=\Kr\ltimes\Kr^*$ the trivial extension of the path algebra $\Kr$ of the Kronecker quiver.
It is well-known that $T(\Kr)$ has a bound quiver presentation with quiver
\begin{center}
\begin{picture}(50,70)
\put(13,19){$1$}
\put(23,32){$\stackrel{x_1}{\longrightarrow}$}
\put(23,21){$\stackrel{y_1}{\longrightarrow}$}
\put(23,10){$\stackrel{x_2}{\longleftarrow}$}
\put(23,0){$\stackrel{y_2}{\longleftarrow}$}
\put(45,19){$2$}
\end{picture}
\end{center}
and relations $x_1y_2=y_1x_2$,
$y_2x_1=x_2y_1$,
$x_2x_1=x_1x_2=y_1y_2=y_2y_1=0$.
The vector space structure of $T(\Kr)$ is that of $\Kr\oplus\Kr^*$ and the multiplication is defined by
$(a,f)(b,g)=(ab,ag + fb)$ for any $a,b\in\Kr~f,g\in\Kr^*$,
here we consider $\Kr^*$ the $\Kr$-$\Kr$-bimodule.
\begin{Lemma}\label{trivial extension}
Let $A:=T(\Kr)$ and $\cL:=\HH^1(A,A)$.
Then $\cL$ is the one-dimensional toral central extension of $\mathfrak{sl}_2$.
In particular, $\mu(\cL)=2$.
\end{Lemma}
\begin{proof}
It is well-known (cf. \cite[(5.5)]{CSS18}) that $\HH^1(\Kr,\Kr)\cong\mathfrak{sl}_2$ and $\cL\cong\mathfrak{gl}_2$,
i.e., $\cL$ is a one-dimensional extension of $\mathfrak{sl}_2$.
We consider the projection map $d:A\rightarrow A;~(a,f)\mapsto(0,f)$.
Clearly, $d^{[p]}=d$.
Moreover, we have
\begin{align*}
d((a,f)(b,g))=&d(ab,ag + fb)=(0,ag+fb)\\
              =&(0,f)(b,g)+(a,f)(0,g)\\
              =&d(a,f)(b,g)+(a,f)d(b,g).
\end{align*}
Consequently, $d\in\Der(A)$. Direct computation shows that $d\notin\IDer(A)$,
so that $\cL\cong\mathfrak{gl}_2\cong\mathfrak{sl}_2\oplus kd$ and $\mu(\cL)=\mu(\mathfrak{sl}_2)+1=2.$
\end{proof}
\begin{Remark}
Since every derivation of the Lie algebra $\mathfrak{sl}_2$ is inner,
it follows that each central extension of $\mathfrak{sl}_2$ splits.
To compute the maximal toral rank,
it depends on the restricted structure of $\mathfrak{gl}_2$.
\end{Remark}

\subsection{Blocks of Frobenius kernels of smooth groups}
Recall that an algebraic group $\cG$ is \textit{smooth} if its coordinate algebra $k[\cG]$ is reduced.
We denote by $k\cG_r$ the group algebra (algebra of distributions) of its $r$-th Frobenius kernel $\cG_r$.
Let $\BB\subseteq k\cG_r$ be a block algebra and we define its complexity $\cx(\BB)$ as the maximum of the complexities of all simple $\BB$-modules.
\begin{Theorem}\label{theorem 2}
Let $\cG$ be a smooth algebraic group with unipotent radical $\cU$ of dimension $n=\dim\cU$,
$\BB\subseteq k\cG_r$ a block algebra and $\cL:=\HH^1(\BB,\BB)$.
Then the following statements hold:
\begin{enumerate}
\item[(1)] If $\BB$ is representation-finite, then $\cL=(0)$ is the trivial Lie algebra,
or $r=1$ and $\cL$ is isomorphic to the Lie algebra in Lemma \ref{h1 of kx/x^p^n},
or the Lie algebra in Theorem \ref{main theorem 1}.\\
\item[(2)] If $\BB$ is tame, then $\cL\cong\mathfrak{gl}_2$ the toral central extension of $\mathfrak{sl}_2$.\\
\item[(3)] In both (1) and (2), we have $\cx(\BB)=\mu(\cL)$.
\end{enumerate}
\end{Theorem}
\begin{proof}
Let $\BB\subseteq k\cG_r$ be a block of finite representation.
Thanks to \cite[Theorem 3.1]{Fa05}, $\BB$ is simple, or $\BB$ is Morita equivalent to $k[X]/(X^{p^n})$ or the group algebra of $\cU\rtimes\GG_{m(1)}$ for some
V-uniserial normal subgroup $\cU$. Hence (1) follows from Lemma \ref{h1 of kx/x^p^n} and Theorem \ref{main theorem 1}.
Our second assertion is a direct consequence of Lemma \ref{trivial extension} and \cite[Theorem 4.6]{Fa05}.
\end{proof}
\begin{Remark}
Given a finite group scheme $\cG$.
Let $\BB\subseteq k\cG$ be a block algebra of the group algebra $k\cG$ and $\cL:=\HH^1(\BB,\BB)$.
In view of Proposition \ref{abelian unipotent group}, Theorem \ref{main theorem 1} and Theorem \ref{theorem 2},
we conjecture that $\cx(\BB)=\mu(\cL)$.
\end{Remark}
\noindent
\textbf{Acknowledgments.}
This work is supported by the NSFC (No. 11801204), NSF of Hubei Province (No. 2018CFB391), and the
Fundamental Research Funds for the Central Universities (No. CCNU18QN033).
I would like to thank Rolf Farnsteiner for his helpful comments.

\end{document}